\documentclass[11pt,a4paper]{article}
\usepackage[margin=2cm]{geometry}
\usepackage{amsmath,amsthm,amsfonts,amssymb,amscd,cite,graphicx}
\usepackage{latexsym}
\usepackage{enumitem}

\usepackage[usenames, dvipsnames]{color}
\definecolor{mygray}{gray}{0.6}

\usepackage{titlesec}
\titleformat{\section}
{\normalfont\fontsize{12}{15}\bfseries}{\thesection}{1em.}{}

\newtheorem{proposition}{Proposition}[section]

\newtheorem{corollary}{Corollary}[section]
\newtheorem{lemma}{Lemma}[section]

\newtheorem{theorem}{Theorem}[section]
\newtheorem{remark}{Remark}[section]

\allowdisplaybreaks[4]

\let\oldbibliography\thebibliography
\renewcommand{\thebibliography}[1]{%
  \oldbibliography{#1}%
  \setlength{\itemsep}{-2pt}%
}

\begin{document}

\baselineskip=0.20in

\makebox[\textwidth]{%
\hglue-15pt
\begin{minipage}{0.6cm}	
\vskip9pt
\end{minipage} \vspace{-\parskip}
\hfill
}
\vskip36pt

\noindent
{\large \bf Statistics on bargraphs of inversion sequences of permutations}\\

\noindent
Toufik Mansour and Mark Shattuck\\

\noindent
\footnotesize {\it Department of Mathematics, University of Haifa, 3498838 Haifa, Israel\\
Email: tmansour@univ.haifa.ac.il}\\

\noindent
\footnotesize {\it Department of Mathematics, University of Tennessee,
37996 Knoxville, TN\\
Email: mshattuc@utk.edu}\\

\setcounter{page}{1} \thispagestyle{empty}

\baselineskip=0.20in

\normalsize

\begin{abstract}
An inversion sequence $(x_1,\ldots,x_n)$ is one such that $1 \leq x_i \leq i$ for all $1 \leq i \leq n$.  We first consider the joint distribution of the area and perimeter statistics on the set $I_n$ of inversion sequences of length $n$ represented as bargraphs.  Functional equations for both the ordinary and exponential generating functions are derived from recurrences satisfied by this distribution.  Explicit formulas are found in some special cases as are expressions for the totals of the respective statistics on $I_n$.  A similar treatment is provided for the joint distribution on $I_n$ for the statistics recording the number of levels, descents and ascents.  Some connections are made between specific cases of this latter distribution and the Stirling numbers of the first kind and Eulerian numbers.
\end{abstract}

\section{Introduction}

Given a permutation $\pi=\pi_1\cdots\pi_n$ of $[n]=\{1,\ldots,n\}$, represented using the one-line notation, the sequence $\textbf{a}=a_1\cdots a_n$ in which $a_i$ records the number of elements of $[i-1]$ occurring to the right of the letter $i$ in $\pi$ for $1 \leq i \leq n$ is called the \emph{inversion table}, or \emph{inversion sequence}, of $\pi$ (see, e.g., \cite[p.~21]{RS}).  For example, $\pi=524613 \in S_6$ has inversion table $\textbf{a}=010242$; note that $0 \leq a_i \leq i-1$ for all $i$.  Conversely, starting with the inversion table $\textbf{a}$, it is seen that one can reconstruct the corresponding permutation $\pi$.  Thus, one may view the inversion table as an alternative representation of the permutation $\pi$.  For our purposes, we will  add $1$ to each entry of $\textbf{a}$ since it will be more convenient to represent the resulting sequence geometrically.

Here, we consider various statistics on sequences $\rho=\rho_1\cdots\rho_n$ of integers satisfying $1 \leq \rho_i \leq i$ for all $i$.  In analogy with avoidance on permutations, the pattern avoidance problem on inversion sequences has been studied from several perspectives, initiated in the papers \cite{MS} and \cite{CMS} concerning the classical avoidance of a single permutation or word pattern of length three.  See, e.g., \cite{KL,Lin,LY,MSa} for extensions of this work in various directions.   Here, we consider new restrictions on inversion sequences obtained in connection with certain statistics on their bargraph representation.

Recall that a bargraph is a self-avoiding random walk in the first quadrant starting at the origin and ending at $(n,0)$ consisting of up $(0,1)$, down $(0,-1)$ and horizontal $(1,0)$ steps.  A sequence $\sigma=\sigma_1\cdots \sigma_n$ of positive integers may be represented as a bargraph $\textbf{b}$ by requiring that the $i$-th column of $\textbf{b}$ contain $\sigma_i$ cells for $1 \leq i \leq n$ (i.e., the height above the $x$-axis of the $i$-th horizontal step is $\sigma_i$). For instance, the permutation $\pi=524613$ above, which has associated inversion sequence $x=121353$ (add $1$ to each entry in the inversion table), may be represented by the bargraph in Figure 1 below.  For examples of recent statistics on bargraphs, see, e.g., \cite{Bl1,Bl2,Bl3} and references contained therein.

In the next section, we consider the joint distribution of the area and perimeter statistics on the set $I_n$ of inversion sequences of length $n$.  We find a recurrence for this distribution on $I_n$ as well as explicit formulas for the total area and perimeter on $I_n$.  In the third section, a comparable treatment is provided for the levels, descents and ascents statistics on $I_n$.  In addition to finding expressions for the totals of these statistics on $I_n$, the distribution is determined explicitly in some specific cases.  We remark that the exponential generating functions of the two joint distributions featured in this paper both satisfy linear first-order functional differential equations with general parameters.  Furthermore, the ordinary generating functions of the distributions can be found explicitly in some general cases by iteration of a functional equation.

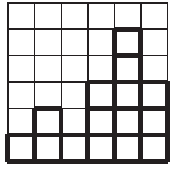
\begin{figure}[htp]
\begin{center}
\begin{picture}(70,60)
\setlength{\unitlength}{.35mm}
\linethickness{0.02mm}
\multiput(0,0)(0,10){7}{\line(2,0){60}}
\multiput(0,0)(10,0){7}{\line(0,2){60}}
\linethickness{0.5mm}
\put(0,0){\line(1,0){60}}
\put(0,10){\line(1,0){60}}
\put(10,20){\line(1,0){10}}
\put(30,20){\line(1,0){30}}\put(30,30){\line(1,0){30}}
\put(40,40){\line(1,0){10}}\put(40,50){\line(1,0){10}}
\put(0,0){\line(0,1){10}}\put(10,0){\line(0,1){20}}
\put(20,0){\line(0,1){20}}\put(30,0){\line(0,1){30}}
\put(40,0){\line(0,1){50}}\put(50,0){\line(0,1){50}}
\put(60,0){\line(0,1){30}}
\end{picture}
\caption{Bargraph of inversion sequence of $\pi=524613 \in S_6$}\label{figaa}
\end{center}
\end{figure}

\section{Area and perimeter  statistics}

In this section, we study the area and perimeter statistics on bargraphs of inversion sequences.  Recall that the area of a bargraph $\lambda$ is that of the first quadrant region subtended by $\lambda$, whereas the perimeter corresponds to the total number of steps of $\lambda$ together with the length of its bottom boundary along the $x$-axis.  It will be convenient to consider the refinement of these statistics to sequences ending in a particular letter.  Given $n \geq 1$ and $1 \leq i \leq n$, let $I_{n,i}$ denote the set of inversion sequences of length $n$ whose last letter is $i$.  Since the perimeter of a bargraph is always even as it includes the bottom boundary, one can consider equivalently the statistic recording half the perimeter, i.e., the semi-perimeter.  Let $a_{n,i}(p,q)$ denote the joint distribution on $I_{n,i}$ for the area and semi-perimeter statistics (marked by $p$ and $q$, respectively).  That is, $a_{n,i}(p,q)=\sum_{\rho \in I_{n,i}}p^{\text{area}(\rho)}q^{\text{sper}(\rho)}$, where $\text{area}(\rho)$ and $\text{sper}(\rho)$ denote respectively the area and semi-perimeter of the bargraph representation of the inversion sequence $\rho$.  For example, $\text{area}(\rho)=15$ and $\text{sper}(\rho)=12$ for the $\rho$ pictured in Figure \ref{figaa}.

The polynomials $a_{n,i}(p,q)$ are determined recursively as follows.

\begin{lemma}\label{lem1}
If $n \geq 2$ and $1 \leq i \leq n$, then
\begin{equation}\label{lem1e1}
a_{n,i}(p,q)=p^iq\sum_{j=i}^{n-1}a_{n-1,j}(p,q)+p^iq\sum_{j=1}^{i-1}q^{i-j}a_{n-1,j}(p,q),
\end{equation}
with $a_{1,1}(p,q)=pq^2$.
\end{lemma}
\begin{proof}
The initial condition follows from the definitions, so assume $n \geq 2$.  To show \eqref{lem1e1}, let $\lambda \in I_{n,i}$ and let $\lambda' \in I_{n-1,j}$ denote the inversion sequence obtained by removing the final column $x$ in the bargraph of $\lambda$.  Note that if $j \in [i,n-1]$, then removing $x$ from $\lambda$ in essence exposes the right boundary of no new cells since the $i$ bottom cells in the last column of $\lambda'$ replace the cells of $x$ in this regard.  Upon taking into account the additional horizontal step, we have that the semi-perimeter increases by one in going from $\lambda'$ to $\lambda$, while the area increases by $i$.  This yields a contribution of $p^iqa_{n-1,j}(p,q)$ for such $\lambda$ towards the overall weight.  Considering all possible $j \in [i,n-1]$ accounts for the first sum on the right side of \eqref{lem1e1}.  On the other hand, if $j\in [i-1]$, then there is in addition an increase of $i-j$ in the semi-perimeter to account for the cells of $x$ that are at a height strictly greater than $j$.  Thus, there are $p^iq^{i-j+1}a_{n-1,j}(p,q)$ such $\lambda$ in this case and considering all $j<i$ accounts for the second sum on the right and completes the proof.
\end{proof}

Let $a_{n,i}=a_{n,i}(p,q)$ and $a_n(y)=a_n(y;p,q)=\sum_{i=1}^na_{n,i}y^i$ for $n \geq 1$.  For example, $a_1(y)=ypq^2$, $a_2(y)=yp^2q^3+y^2p^3q^4$, $a_3(y)=yp^3q^4(1+pq)+y^2p^4q^5(1+p)+y^3p^5q^6(1+p)$.

Multiplying both sides of \eqref{lem1e1} by $y^i$, summing over $1 \leq i \leq n$ and interchanging summation yields
\begin{align}
a_n(y)&=q\sum_{i=1}^n(yp)^i\sum_{j=i}^{n-1}a_{n-1,j}+q\sum_{i=1}^n(yp)^i\sum_{j=1}^{i-1}q^{i-j}a_{n-1,j}\notag\\
&=q\sum_{j=1}^{n-1}a_{n-1,j}\sum_{i=1}^j(yp)^i+q\sum_{j=1}^{n-1}a_{n-1,j}q^{-j}\sum_{i=j+1}^n(ypq)^i\notag\\
&=q\sum_{j=1}^{n-1}a_{n-1,j}\cdot\frac{yp-(yp)^{j+1}}{1-yp}+q\sum_{j=1}^{n-1}a_{n-1,j}q^{-j}\cdot \frac{(ypq)^{j+1}-(ypq)^{n+1}}{1-ypq}\notag\\
&=\frac{ypq}{1-yp}(a_{n-1}(1)-a_{n-1}(yp))+\frac{ypq^2}{1-ypq}\left(a_{n-1}(yp)-(ypq)^na_{n-1}(1/q)\right), \qquad n \geq 2,\label{perime2}
\end{align}
where we have used the fact $a_n(1)=\sum_{i=1}^na_{n,i}$.

Define the exponential generating function $f(x,y)=f(x,y;p,q)$ by
$$f(x,y)=\sum_{n\geq 1}a_n(y)\frac{x^n}{n!}.$$
Multiplying both sides of \eqref{perime2} by $\frac{x^{n-1}}{(n-1)!}$, and summing over $n \geq 2$, implies $f(x,y)$ satisfies the following linear functional differential equation

\begin{theorem}
We have
\begin{align}
\frac{\partial}{\partial x}f(x,y)&=ypq^2+\frac{ypq}{1-yp}(f(x,1)-f(x,yp))+\frac{ypq^2}{1-ypq}(f(x,yp)-ypqf(xypq,1/q)).\label{fxye1}
\end{align}
\end{theorem}

Note that \eqref{fxye1} with $p=q=1$ gives
\begin{align*}
\frac{\partial}{\partial x}f(x,y;1,1)&=y+\frac{y}{1-y}f(x,1;1,1)-\frac{y^2}{1-y}f(xy,1;1,1).
\end{align*}
From this, one can obtain
\begin{equation}\label{fxy11}
f(x,y;1,1)=\frac{y}{1-y}\ln\left(\frac{1-xy}{1-x}\right).
\end{equation}
In particular $f(x,1;1,1)=\frac{x}{1-x}$.

To determine a formula for the total area of all members of $I_n$, let us consider the derivative with respect
to $p$ evaluated at $p=1$.  Define $Pf(x,y)=\frac{\partial}{\partial p}f(x,y;p,1)\mid_{p=1}=\sum_{m\geq1}Pf_m(y)\frac{x^m}{m!}$. Then \eqref{fxye1}, taken together with \eqref{fxy11}, gives
\begin{align*}
\frac{\partial}{\partial x}Pf(x,y)&=\frac{y}{1-y}Pf(x,1)-\frac{y^2}{1-y}Pf(xy,1)+\frac{y}{(1-x)(1-xy)^2}.
\end{align*}
Thus, $Pf_{m+1}(y)=\frac{y-y^{m+2}}{1-y}Pf_m(1)+m!\sum_{i=1}^{m+1}iy^i$ for $m \geq 0$, with $Pf_0(y)=0$. By induction on $m$, we have
$$Pf_m(1)=\frac{m!}{2}\left(\binom{m+2}{2}-1\right),$$
which leads to
$$Pf_{m+1}(y)=m!\frac{y-y^{m+2}}{2(1-y)}\left(\binom{m+2}{2}-1\right)+m!\sum_{i=1}^{m+1}iy^i.$$
From this, one can find $\sum_{m\geq1}Pf_m(y)\frac{x^m}{m!}$ explicitly.

\begin{corollary}
The generating function $Pf(x,y)$ for the sum of the areas of all members of $I_{n,j}$ for $n \geq 1$ and $1 \leq j \leq n$ is given by
\begin{align*}
&\frac{y(1+y)}{2(1-y)^2}\ln\left(\frac{1-xy}{1-x}\right)
+\frac{(4x^3y^2-8x^2y^2-4x^2y+5xy^2+8xy-x-6y+2)xy}{4(1-xy)^2(1-x)^2(1-y)}\\
&=yx+y(3y+2)\frac{x^2}{2!}+y(11y^2+9y+7)\frac{x^3}{3!}+3y(17y^3+15y^2+13y+11)\frac{x^4}{4!}+\cdots.
\end{align*}
In particular, the sum of the areas of all members of $I_n$ is given by $\frac{n!}{2}\left(\binom{n+2}{2}-1\right)$.
\end{corollary}

Now let us consider the derivative at $q=1$.  Define $Qf(x,y)=\frac{\partial}{\partial q}f(x,y;1,q)\mid_{q=1}=\sum_{m\geq1}Qf_m(y)\frac{x^m}{m!}$. Then \eqref{fxye1}, together with \eqref{fxy11}, gives
\begin{align*}
\frac{\partial}{\partial x}Qf(x,y)&=\frac{y}{1-y}Qf(x,1)-\frac{y^2}{1-y}Qf(xy,1)+\frac{y^2}{(1-y)^3}\ln\left(\frac{1-xy}{1-x}\right)\\
&-\frac{(2(x-2)-2(2x^2-x-4)y+2(x+2)(x^2-x-1)y^2-x(x^2+x-4)y^3+x^2(x-1)y^4)y}{2(1-y)^2(1-x)(1-xy)^2}.
\end{align*}
By finding the coefficient of $x^m/m!$, and taking the limit as $y\rightarrow1$, we have
$$Qf_{m+1}(1)=(m+1)Qf_m(1)+\frac{1}{6}(m+8)(m+1)!, \qquad m \geq 1,$$
with $Qf_1(1)=2$. Hence, by induction on $m$, we obtain
$$Qf_m(1)=\frac{1}{12}(m^2+15m+8)m!,$$ which implies
the following result.

\begin{corollary}
The sum of the semi-perimeters of all members of $I_n$ for $n \geq 1$ is given by
$$\frac{1}{12}(n^2+15n+8)n!.$$
\end{corollary}

We have the following sign balance result for the area and semi-perimeter statistics on $I_n$.

\begin{proposition}
If $n \geq 3$, then
\begin{equation}\label{p=-1}
a_n(y;-1,1)=0
\end{equation}
and
\begin{equation}\label{q=-1}
a_n(y;1,-1)=2^{n-2}y^{n-1}(y-1).
\end{equation}
\end{proposition}
\begin{proof}
Formulas \eqref{p=-1} and \eqref{q=-1} can be obtained from \eqref{perime2} (and the equation directly prior) by an induction argument.  Here, we provide a direct bijective proof.  Let $\rho=\rho_1\cdots \rho_n \in I_n$ where $n \geq 3$. Note first that replacing $\rho_2$ with $3-\rho_2$ changes the area of the bargraph of $\rho$ by one for all $\rho$, which implies \eqref{p=-1}. Now let $k$ be the smallest index $i$, if it exists, such that $\rho_i \notin \{i-1,i\}$; note that $k \geq 3$.  Within $\rho$, consider replacing $\rho_{k-1}$ with $2k-3-\rho_{k-1}$ to obtain $\rho' \in I_n$.  Let $m=\max\{\rho_{k-2},\rho_k\}$.  Note that $\rho_{k-1}=m$ implies $2k-3-\rho_{k-1}>m$ and $\rho_{k-1}>m$ implies $2k-3-\rho_{k-1}\geq m$.  Since $\rho_{k-1} \geq m$, it follows that the bargraphs of $\rho$ and $\rho'$ have semi-perimeters differing by one and hence are of opposite parity.  Thus, the mapping $\rho \mapsto \rho'$ provides a sign-changing involution on all of $I_n$ for which $k$ is defined.

Note that this mapping is defined on the entirety of $I_{n,j}$ if $j \in [n-2]$ since $k$ is guaranteed to exist in this case.  If $\rho=\rho_1\cdots \rho_n \in I_{n,j}$ for $j=n-1$ or $n$, then no such $k$ exists if and only if $\rho_i \in \{i-1,i\}$ for $2 \leq i \leq n-1$.  Thus, there are $2^{n-2}$ possible members of $I_{n,j}$ in either case.  Finally, each possible member of $I_{n,n-1}$ and $I_{n,n}$ is (weakly) increasing and hence has semi-perimeter $2n-1$ or $2n$, respectively, which explains the signs and completes the proof of \eqref{q=-1}.
\end{proof}

The array $a_{n,i}$ may also be determined by the following three-term recurrence.

\begin{proposition}\label{perprop}
If $n \geq 3$, then
\begin{equation}\label{aeq2}
a_{n,i}=p(q+1)a_{n,i-1}-p^2qa_{n,i-2}+p^iq(q-1)a_{n-1,i-1}, \qquad 3 \leq i \leq n,
\end{equation}
with $a_{n,1}=pqa_{n-1}(1)$ and $a_{n,2}=pa_{n,1}+p^2q(q-1)a_{n-1,1}$ for $n \geq 2$ and $a_{1,1}=pq^2$.
\end{proposition}
\begin{proof}
Consider the difference $a_{n,i}-pa_{n,i-1}$, which by \eqref{lem1e1} is given by
\begin{align}
a_{n,i}-pa_{n,i-1}&=-p^iqa_{n-1,i-1}+p^iq\sum_{j=1}^{i-1}q^{i-j}a_{n-1,j}-p^iq\sum_{j=1}^{i-2}q^{i-j-1}a_{n-1,j}\notag\\
&=p^iq(q-1)\sum_{j=1}^{i-1}q^{i-j-1}a_{n-1,j}, \qquad 2 \leq i \leq n. \label{aeq3}
\end{align}
By \eqref{aeq3}, we then have for $3 \leq i \leq n$,
$$(a_{n,i}-pa_{n,i-1})-pq(a_{n,i-1}-pa_{n,i-2})=p^iq(q-1)\left(\sum_{j=1}^{i-1}q^{i-j-1}a_{n-1,j}-\sum_{j=1}^{i-2}q^{i-j-1}a_{n-1,j}\right),$$
so that
$$a_{n,i}-p(q+1)a_{n,i-1}+p^2qa_{n,i-2}=p^iq(q-1)a_{n-1,i-1},$$
which gives \eqref{aeq2}.  The initial condition for $a_{n,i}$ when $i=1$ follows from the definitions, upon appending 1 to any member of $I_{n-1}$.  Taking $i=2$ in \eqref{aeq3} gives the formula for $a_{n,2}$.
\end{proof}

We conclude this section by finding an expression for the ordinary generating function of $a_n(y)$ when $q=1$. Define $A(x,y)=A(x,y;p,q)=\sum_{n\geq1}a_n(y)x^n$.
By \eqref{perime2}, we have
\begin{align}\label{eqPP1}
A(x,y)&=xypq^2+\frac{xypq}{1-yp}(A(x,1)-A(x,yp))
+\frac{xypq^2}{1-ypq}\left(A(x,yp)-ypqA(xypq,1/q)\right).
\end{align}
By \eqref{eqPP1} with $y=q=1$, we get
$$A(x,1;p,1)=\frac{xp(1-p)}{1-p-xp}-\frac{xp^2}{1-p-xp}A(xp,1;p,1).$$
Iterating the last expression (where it is assumed $|x|,|p|<1$) yields
\begin{equation}\label{ogfa(n)}
A(x,1;p,1)=x(1-p)\sum_{j\geq0}\frac{(-1)^jx^jp^{j+\binom{j+2}{2}}}
{\prod_{i=0}^{j}(1-p-xp^{i+1})}.
\end{equation}
By \eqref{eqPP1} with $q=1$, we have
$$A(x,y;p,1)=xyp+\frac{xyp}{1-yp}(A(x,1;p,1)-ypA(xyp,1;p,1)),$$
which by \eqref{ogfa(n)} implies the following result.
\begin{theorem}\label{ogfa(n)t}
The (ordinary) generating function $\sum_{n\geq 1}a_n(y;p,1)x^n$ is given by
\begin{equation}\label{ogfa(n)te1}
A(x,y;p,1)=xyp+\frac{x^2yp(1-p)}{1-yp}\left(\sum_{j\geq0}\frac{(-1)^jx^jp^{j+\binom{j+2}{2}}}
{\prod_{i=0}^{j}(1-p-xp^{i+1})}-y^2p^2\sum_{j\geq0}\frac{(-1)^jx^jy^jp^{2j+\binom{j+2}{2}}}
{\prod_{i=0}^{j}(1-p-xyp^{i+2})}\right).
\end{equation}
\end{theorem}

\section{Levels, descents and ascents}

Recall that a \emph{level}, \emph{descent} or \emph{ascent} within a word $w=w_1w_2\cdots$ is an index $i$ such that $w_i=w_{i+1}$, $w_i>w_{i+1}$ or $w_i<w_{i+1}$, respectively. Given $n \geq 1$ and $1 \leq i \leq n$, let $b_{n,i}(p,q,r)$ denote the joint distribution for the level, descent and ascent statistics on $I_{n,i}$, marked by $p$, $q$ and $r$, respectively.  Considering whether the penultimate letter $j$ of a member of $I_{n,i}$ for $1 \leq i \leq n-1$ is equal to, greater than or less than $i$ yields the following recurrence for $b_{n,i}(p,q,r)$, where the condition for $i=n$ follows from observing that all members of $I_{n,n}$ must end in an ascent.
\begin{lemma}\label{ldalem}
If $n \geq 2$ and $1 \leq i \leq n-1$, then
\begin{equation}\label{ldaleme1}
b_{n,i}(p,q,r)=pb_{n-1,i}(p,q,r)+q\sum_{j=i+1}^{n-1}b_{n-1,j}(p,q,r)+r\sum_{j=1}^{i-1}b_{n-1,j}(p,q,r),
\end{equation}
with $b_{n,n}(p,q,r)=r\sum_{i=1}^{n-1}b_{n-1,i}(p,q,r)$ for $n \geq 2$ and $b_{1,1}(p,q,r)=1$.
\end{lemma}

Let $b_{n,i}=b_{n,i}(p,q,r)$ and $b_n(y)=b_n(y;p,q,r)$ be given by $b_{n}(y)=\sum_{i=1}^nb_{n,i}y^i$ for $n \geq 1$.
For example, $b_1(y)=y$, $b_2(y)=yp+y^2r$, $b_3(y)=y(qr+p^2)+2y^2pr+y^3r(p+r)$.

By \eqref{ldaleme1}, we have
\begin{align}
b_n(y)&-y^nrb_{n-1}(1)=p\sum_{i=1}^{n-1}b_{n-1,i}y^i+q\sum_{i=1}^{n-2}y^i\sum_{j=i+1}^{n-1}b_{n-1,j}+r\sum_{i=2}^{n-1}y^i\sum_{j=1}^{i-1}b_{n-1,j}\notag\\
&=pb_{n-1}(y)+q\sum_{j=1}^{n-1}b_{n-1,j}\sum_{i=1}^{j-1}y^i+r\sum_{j=1}^{n-1}b_{n-1,j}\sum_{i=j+1}^{n-1}y^i \label{ldae0}\\
&=pb_{n-1}(y)+\frac{q}{1-y}\left(yb_{n-1}(1)-b_{n-1}(y)\right)+\frac{yr}{1-y}\left(b_{n-1}(y)-y^{n-1}b_{n-1}(1)\right),\notag
\end{align}
which may be rewritten as
\begin{equation}\label{ldae1}
b_n(y)=\left(p+\frac{yr-q}{1-y}\right)b_{n-1}(y)+\frac{y(q-y^nr)}{1-y}b_{n-1}(1), \qquad n \geq 2.
\end{equation}

Let $H_n=\sum_{i=1}^n\frac{1}{i}$ denote the $n$-th harmonic number.

\begin{corollary}\label{totals}
The total number of levels, descents and ascents in all members of $I_n$ for $n \geq 1$ is given by $n!(H_n-1)$, $\frac{1}{2}(n+1)!-n!H_n$ and $\frac{n-1}{2}n!$, respectively.
\end{corollary}
\begin{proof}
By differentiating both sides of \eqref{ldae1} with respect to $p$ when $y=q=r=1$, setting $p=1$ and making use of the fact $b_n(y)=n!$ when all arguments are unity, we have
$\frac{\partial}{\partial p}b_n(1)\mid_{p=q=r=1}=
(n-1)!+n\frac{\partial}{\partial p}b_{n-1}(1)\mid_{p=q=r=1}$ for $n \geq 2$, with $\frac{\partial}{\partial p}b_1(1)\mid_{p=q=r=1}=0$. By induction, this yields $\frac{\partial}{\partial p}b_n(1)\mid_{p=q=r=1}=n!\sum_{j=2}^n\frac{1}{j}=n!(H_n-1)$, which implies the first formula.
Differentiating \eqref{ldae1} at $q=1$ with $y=p=r=1$, we have in a similar fashion
$\frac{\partial}{\partial q}b_n(1)\mid_{p=q=r=1}=
\frac{n-2}{2}(n-1)!+n\frac{\partial}{\partial q}b_{n-1}(1)\mid_{p=q=r=1}$ for $n \geq 2$, with $\frac{\partial}{\partial q}b_n(1)\mid_{p=q=r=1}=0$. By induction, we get $\frac{\partial}{\partial q}b_n(1)\mid_{p=q=r=1}=\frac{1}{2}(n+1)!-n!H_n$.
Finally, differentiating \eqref{ldae1} with respect to $r$ at $r=1$ with $y=p=q=1$ yields
$\frac{\partial}{\partial r}b_n(1)\mid_{p=q=r=1}=\frac{n-1}{2}n!$, which gives the last formula and completes the proof.
\end{proof}

Let $g(x,y)=g(x,y;p,q,r)$ be given by $g(x,y)=\sum_{n\geq 1}b_n(y)\frac{x^n}{n!}$.  Then \eqref{ldae1} may be rewritten in terms of generating functions as follows.

\begin{theorem}\label{ldath}
We have
\begin{align}
\frac{\partial}{\partial x}g(x,y)&=y+\left(p+\frac{yr-q}{1-y}\right)g(x,y)+\frac{y}{1-y}\left(qg(x,1)-yrg(xy,1)\right).\label{gxye1}
\end{align}
\end{theorem}

Note that \eqref{gxye1} with $p=q=r=1$ gives
\begin{align*}
\frac{\partial}{\partial x}g(x,y;1,1,1)&=y+\frac{y}{1-y}g(x,1;1,1,1)-\frac{y^2}{1-y}g(xy,1;1,1,1),
\end{align*}
which implies $$g(x,y;1,1,1)=\frac{y}{1-y}\ln\left(\frac{1-xy}{1-x}\right),$$
as expected, since $f(x,y)$ and $g(x,y)$ agree when all other arguments are unity.  In particular,  $g(x,1;1,1,1)=\frac{x}{1-x}$.

One may extend the results of Corollary \ref{totals} above as follows by making use of \eqref{gxye1}.

\begin{theorem}
The generating functions for the totals of the levels, descents and ascents statistics on $I_{n,j}$ for $n \geq 1$ and $1 \leq j \leq n$ are given respectively
by
\begin{equation}\label{tote1}
\frac{\partial}{\partial p}g(x,y;p,1,1)\mid_{p=1}=xy+\frac{2(xy-y-1)\ln(1-xy)-2y(x-2)\ln(1-x)-y(\ln^2(1-xy)-\ln^2(1-x))}{2(1-y)},
\end{equation}
\begin{align}
\frac{\partial}{\partial q}g(x,y;1,q,1)\mid_{q=1}&=\frac{xy}{2(1-y)}\left(\frac{3x-2}{1-x}-\frac{xy^2}{1-xy}\right)+\frac{(1-xy)\ln(1-xy)-y(2-x-y)\ln(1-x)}{(1-y)^2}\notag\\
&\quad+\frac{y(\ln^2(1-xy)-\ln^2(1-x))}{2(1-y)},\label{tote2}
\end{align}
and
\begin{equation}\label{tote3}
\frac{\partial}{\partial r}g(x,y;1,1,r)\mid_{r=1}=\frac{xy}{2(1-y)}\left(\frac{2-x}{1-x}-\frac{xy^2}{1-xy}\right)+\frac{y(1-xy)}{(1-y)^2}\ln\left(\frac{1-x}{1-xy}\right).
\end{equation}
\end{theorem}
\begin{proof}
Let $h(x,y)=\frac{\partial}{\partial p}g(x,y;p,1,1)\mid_{p=1}$ and note by \eqref{gxye1} that $h(x,y)$ must satisfy
\begin{equation}\label{hxye1}
\frac{\partial}{\partial x}h(x,y)=g(x,y;1,1,1)+\frac{y}{1-y}h(x,1)-\frac{y^2}{1-y}h(xy,1).
\end{equation}
To show that the purported formula for $h(x,y)$ indeed satisfies \eqref{hxye1}, one must write $h(x,1)=\lim_{z\rightarrow 1}h(x,z)$ to evaluate the right-hand side.  Upon observing
$$\lim_{z\rightarrow 1}\left(\frac{\ln^2(1-xz)-\ln^2(1-x)}{1-z}\right)=\frac{2x\ln(1-x)}{1-x}$$
and
$$\lim_{z\rightarrow1}\left(\frac{(1+z-xz)\ln(1-xz)-z(2-x)\ln(1-x)}{1-z}\right)=\ln(1-x)+\frac{x(2-x)}{1-x},$$
and recalling the expression for $g(x,y;1,1,1)$, one gets
$$\frac{y}{1-y}\ln\left(\frac{1-xy}{1-x}\right)-\frac{xy+y\ln(1-x)}{(1-x)(1-y)}+\frac{xy^3+y^2\ln(1-xy)}{(1-y)(1-xy)}$$
for the right side of \eqref{hxye1}.  This is seen to coincide with $\frac{\partial}{\partial x}h(x,y)$, which completes the proof of \eqref{tote1}.  Similar proofs apply to \eqref{tote2} and \eqref{tote3}. Note that for \eqref{tote2}, it is convenient to first write the formula for $\frac{\partial}{\partial q}g(x,y;1,q,1)\mid_{q=1}$ as
\begin{align*}
&\frac{x^2y}{2(1-y)}\left(\frac{1}{1-x}-\frac{y^2}{1-xy}\right)\\
&+\frac{y(\ln^2(1-xy)-\ln^2(1-y))}{2(1-y)}+\frac{(1-xy)\ln(1-xy)-y(2-x-y)\ln(1-x)-xy(1-y)}{(1-y)^2}
\end{align*}
and compute three separate limits as $y \rightarrow 1$, where the third expression requires two applications of L'Hopital's rule.  For \eqref{tote3}, it is best to write the formula for $\frac{\partial}{\partial r}g(x,y;1,1,r)\mid_{r=1}$ as
\begin{align*}
&\frac{x^2y}{2(1-y)}\left(\frac{1}{1-x}-\frac{y^2}{1-xy}\right)+\frac{y(1-xy)(\ln(1-x)-\ln(1-xy))+xy(1-y)}{(1-y)^2}
\end{align*}
prior to computing the limit.
\end{proof}

Let $c(n,k)$ for $1 \leq k \leq n$ and $e(n,k)$ for $ 0 \leq k \leq n-1$ denote the (signless) Stirling number of the first kind and Eulerian number, respectively.  Recall that the number of permutations of $[n]$ with $k$ cycles and $k$ ascents is given by $c(n,k)$ and $e(n,k)$, respectively; see, e.g., \cite[Sections~6.1~and~6.2]{GKP}.  The following result provides a connection between inversion sequences and the Stirling and Eulerian numbers.

\begin{theorem}\label{gth2}
If $n \geq 1$, then
\begin{equation}\label{gth2e1}
b_n(1;t,1,1)=\sum_{k=0}^{n-1}c(n,k+1)t^k
\end{equation}
and
\begin{equation}\label{gth2e2}
b_n(1;1,1,t)=b_n(1;t,t,1)=\sum_{k=0}^{n-1}e(n,k)t^k.
\end{equation}
\end{theorem}
\begin{proof}
Formulas \eqref{gth2e1} and \eqref{gth2e2} can be obtained by taking $p=t$, $r=t$ or $p=q=t$ in \eqref{ldae0} (with all other arguments equal to unity) and comparing the resulting equations with the known recurrences for the distributions $\sum_{k=0}^{n-1}c(n,k+1)t^k$ and $\sum_{k=0}^{n-1}e(n,k)t^k$.  We leave the details to the reader.  However, we find it more instructive to provide direct bijective proofs of these formulas as follows, starting with \eqref{gth2e1}.  Let $I_n(k)$ denote the subset of $I_n$ whose members have exactly $k$ levels and $S_n(k)$ the subset of $S_n$ whose members have exactly $k$ cycles.  Given $\rho=\rho_1\cdots\rho_n \in I_n(k)$, we generate a member of $S_n(k+1)$ as follows.  First let $\pi_1=(1)$ and we subsequently form permutations $\pi_2,\ldots,\pi_n$ by successively inserting the elements $2,\ldots,n$ into cycles.  Let $\ell=\rho_{j-1}$ where $2 \leq j \leq n$.  If $\rho_{j}\in[j]-\{\ell\}$ and is the $i$-th smallest member of this set, to obtain $\pi_j$ in this case, we insert the element $j$ so that it directly follows $i$ in its current cycle within $\pi_{j-1}$ (expressed in standard cycle form).  If $\rho_j=\ell$, in which case the $(j-1)$-st and $j$-th letters of $\rho$ correspond to a level, then insert the element $j$ into a new cycle by itself to obtain $\pi_j$.  Let $f(\rho)=\pi_n \in S_n(k+1)$.  For example, if $n=8$, $k=3$ and $\rho=12243377 \in I_{8}(3)$, then $\pi_1=(1)$, $\pi_2=(1,2)$, $\pi_3=(1,2),(3)$, $\pi_4=(1,2),(3,4)$, $\pi_5=(1,2),(3,5,4)$, $\pi_6=(1,2),(3,5,4),(6)$, $\pi_7=(1,2),(3,5,4),(6,7)$ so that $f(\rho)=\pi_8=(1,2),(3,5,4),(6,7),(8) \in S_8(4)$.  One may verify that $f$ is a bijection between $I_n(k)$ and $S_n(k+1)$ for $0 \leq k \leq n-1$, which implies \eqref{gth2e1}.

Given $\rho=\rho_1\cdots\rho_n \in I_n$, let $\rho'$ be obtained from $\rho$ by replacing $\rho_i$ with $i+1-\rho_i$ for $1 \leq i \leq n$.  Then it is seen that the mapping $\rho \mapsto \rho'$ is a bijection on $I_n$ which demonstrates that the ascents statistic distribution is equal to the distribution for the sum of the number of levels and descents.  Thus, to complete the proof of \eqref{gth2e2}, we need only show that $|I_n^{(k)}|=|S_n^{(k)}|$ for $0 \leq k \leq n-1$, where $I_n^{(k)}$ denotes the subset of $I_n$ and $S_n^{(k)}$ the subset of $S_n$ whose members have $k$ ascents.  Let $\rho=\rho_1\cdots\rho_n \in I_n^{(k)}$, and first write $\pi_1=1$.  If $\rho_j=\ell$ where $2 \leq j \leq n$, then to obtain $\pi_j\in S_j$, we append the letter $\ell$ to $\pi_{j-1}$ and increase all letters in $[\ell,j-1]$ by $1$. Note that the final two letters of $\pi_j$ form an ascent if and only if $\rho_{j-1}\rho_j$ within $\rho$ does.  Let $g(\rho)=\pi_n$ denote the member of $S_n$ that results after all letters $\ell$ have been appended successively as described.  For example, if $n=8$, $k=4$ and $\rho=12142473 \in I_8^{(4)}$, then $\pi_1=1$, $\pi_2=12$, $\pi_3=231$, $\pi_4=2314$, $\pi_5=34152$, $\pi_6=351624$, $\pi_7=3516247$ and thus $g(\rho)=\pi_8=46172583\in S_8^{(4)}$.  Note that since previous ascents are preserved in later steps, the permutation $\pi_n$ has the same number of ascents as $\rho$.  One may then verify that $g$ is a bijection between $I_n^{(k)}$ and $S_n^{(k)}$ for each $0 \leq k \leq n-1$, which completes the proof.
\end{proof}

\begin{remark} Note that members of $S_n$ starting with $r$ correspond (under $g^{-1}$) to members $\rho$ of $I_n$ wherein the longest subsequence of $\rho$ that is itself an inversion sequence (when its subscripts are re-indexed to start $1,2,\ldots$) and not containing the initial $1$ has length $r-1$.
\end{remark}

\begin{remark} The polynomial $b_n(1;t,1,t)$ was shown in \cite[Section~2.3.4]{Par} by a combinatorial argument to coincide with sequence A122890 in \cite{Sloane}.
\end{remark}

Using the interpretation implicit in \eqref{gth2e1}, one can provide a combinatorial proof of the formulas above for the totals of the levels, descents and ascents statistics on $I_n$ as follows. \medskip

\noindent{\textbf{Combinatorial proof of Corollary \ref{totals}.}\medskip

By \eqref{gth2e1}, the total number of levels in all members of $I_n$ equals the total number of cycles in $S_n$ minus $n!$.  By \cite[Theorem~12]{BQ}, which was provided a combinatorial proof, the average number of cycles in $S_n$ is given by $H_n$, which implies the first formula in Corollary \ref{totals}.   The bijection $\rho\mapsto\rho'$ from the proof of Theorem \ref{gth2} above shows that the ascents at index $i$ within $I_n$ for each $i$ are equal in number to the union of all levels and descents at index $i$.  As there are $(n-1)n!$ non-terminal positions within all members of $I_n$, each of which corresponds to either a level, descent or ascent, it follows that the total number of ascents in $I_n$ is given by $\frac{n-1}{2}n!$.  Finally, for descents, we subtract the first and last expressions in Corollary \ref{totals} from $(n-1)n!$ to get a total of
$$(n-1)n!-\left(n!(H_n-1)+\frac{n-1}{2}n!\right)=\frac{1}{2}(n+1)!-n!H_n,$$
which completes the proof. \hfill \qed \medskip

Let $e_n(y)$ and $o_n(y)$ denote the restriction of $b_n(y)$ to those members of $I_n$ having an even or an odd number of levels, respectively.  Note that $b_n(y;-p,q,r)=e_n(y;p,q,r)-o_n(y;p,q,r)$.  We have the following sign balance result concerning the levels statistic on $I_n$.

\begin{proposition}\label{balance}
If $n \geq 2$, then
\begin{equation}\label{bale1}
b_n(y;-t,t,t)=(-1)^n2^{n-2}t^{n-1}(y-1)y.
\end{equation}
In particular, the statistic on $I_n$ recording the number of levels is balanced for $n \geq 2$.
\end{proposition}
\begin{proof}
The latter statement follows from the former upon taking $y=1$, so we need only prove the former.  Note that one can show \eqref{bale1} by an induction argument using \eqref{ldae0} with $q=r=t=-p$, the details of which we leave to the reader.  Here, we wish to provide a direct bijective proof of \eqref{bale1} which makes use of a sign-changing involution.  Suppose $\rho=\rho_1\cdots\rho_n \in I_n$ contains at least one letter greater than 2.  Let $j_0$ be the minimal $j$ such that $\rho_j>2$.  Let $\rho^*$ be obtained from $\rho$ by replacing $\rho_{j_0-1}$ with $3-\rho_{j_0-1}$, which can be done since $j_0\geq 3$ ensures that it is not the first letter that is undergoing this replacement (which of course would not be allowed as $\rho_1=1$ for all $\rho$).  Note that $\rho^*$ has a parity with respect to the levels statistic opposite to that of $\rho$.  Thus, $\rho \mapsto \rho^*$  is a sign-changing involution that is defined on the entirety of $I_n$ except for the subset consisting of its binary members, which we will denote by $B_n$.

We now determine the (signed) weight of the members of $B_n$.  First suppose $\lambda=\lambda_1\cdots \lambda_n\in B_n$, with $\lambda_n=1$.  Then since $\lambda$ is binary, there must be an even number $m$ of indices $i \in [n-1]$ such that $\lambda_i\neq \lambda_{i+1}$ since $\lambda_1=\lambda_n=1$.  Thus, the number of levels is given by $n-1-m$ for all such $\lambda$.  In particular, all $\lambda \in B_n$ with $\lambda_n=1$ have levels parity equal to that of $n-1$.  Since there are clearly $2^{n-2}$ such $\lambda$, their (signed) weight is given by $(-t)^{n-1}2^{n-2}y$.  By similar reasoning, the weight of all $\lambda \in B_n$ with $\lambda_n=2$ is given $-(-t)^{n-1}2^{n-2}y^2$.  Combining this case with the prior one yields formula \eqref{bale1}.
\end{proof}

Considering the difference $b_{n,i}-b_{n,i-1}$, and using \eqref{ldaleme1}, yields the following three-term recurrence for the array $b_{n,i}$.
\begin{proposition} If $n \geq 3$, then
\begin{equation}\label{beq2}
b_{n,i}=b_{n,i-1}+(p-q)b_{n-1,i}+(r-p)b_{n-1,i-1}, \qquad 2 \leq i \leq n-1,
\end{equation}
with $b_{n,1}=(p-q)b_{n-1,1}+qb_{n-1}(1)$ and $b_{n,n}=rb_{n-1}(1)$ for $n \geq 2$ and $b_{1,1}=1$.
\end{proposition}

We conclude by finding an expression for the ordinary generating function of the polynomial $b_n(y)$ when $y=1$.  To do so, first define $B_n(v)=\sum_{i=1}^nb_{n,i}v^{i-1}$. Note that
$$\sum_{i=2}^{n-1}b_{n,i-1}v^{i-1}=\sum_{i=1}^{n-2}b_{n,i}v^i=v(B_n(v)-b_{n,n}v^{n-1}-b_{n,n-1}v^{n-2}).$$
By \eqref{ldaleme1} with $i=n-1$, we have
$$b_{n,n-1}=pb_{n-1,n-1}+r(B_{n-1}(1)-b_{n-1,n-1})=rB_{n-1}(1)+r(p-r)B_{n-2}(1)$$
so that
$$\sum_{i=2}^{n-1}b_{n,i-1}v^{i-1}=vB_n(v)-rv^{n-1}(1+v)B_{n-1}(1)-r(p-r)v^{n-1}B_{n-2}(1).$$
Multiplying both sides of \eqref{beq2} by $v^{i-1}$, and summing over $2 \leq i \leq n-1$, then implies
\begin{align*}
B_n(v)-rv^{n-1}B_{n-1}(1)-B_n(0)&=vB_n(v)-rv^{n-1}(1+v)B_{n-1}(1)-r(p-r)v^{n-1}B_{n-2}(1)\\
&\quad+(p-q)(B_{n-1}(v)-B_{n-1}(0))+(r-p)(vB_{n-1}(v)-rv^{n-1}B_{n-2}(1)),
\end{align*}
which may be rewritten as
\begin{equation}\label{Bneqn}
(1-v)B_n(v)=B_n(0)-(p-q)B_{n-1}(0)-rv^nB_{n-1}(1)+(p-q+(r-p)v)B_{n-1}(v), \qquad n \geq 3,
\end{equation}
with $B_n(0)=(p-q)B_{n-1}(0)+qB_{n-1}(1)$ for $n \geq 2$.  Note that \eqref{Bneqn} is also seen to hold when $n =2$.

Let $B(x,v)=\sum_{n\geq1}B_n(v)x^n$.  Then \eqref{Bneqn} may be rewritten as
\begin{equation}\label{Bxveq}
(1-v)B(x,v)=-xv+(1-(p-q)x)B(x,0)-xrvB(xv,1)+x(p-q+(r-p)v)B(x,v),
\end{equation}
with $B(x,0)=x+(p-q)xB(x,0)+xqB(x,1)$.  Applying the kernel method \cite{HT}, and letting $v=\rho(x)=\frac{1-(p-q)x}{1-(p-r)x}$ in \eqref{Bxveq}, we obtain
$$B(x,1)=\frac{\rho(x)-1}{q}+\frac{r\rho(x)}{q}B(x\rho(x),1).$$
Iteration of this last formula leads to the following result.
\begin{theorem}
The (ordinary) generating function $\sum_{n\geq1}b_{n}(1;p,q,r)x^n$ is given by
\begin{equation}\label{ogf}
B(x,1)=\sum_{j\geq0}\frac{r^j(v_j(x)-1)}{q^{j+1}}v_0(x)v_1(x)\cdots v_{j-1}(x),
\end{equation}
where $v_j(x)$ is defined recursively by $v_j(x)=\rho(xv_{j-1}(x))$ for $j \geq 1$ with $v_0(x)=\rho(x)=\frac{1-(p-q)x}{1-(p-r)x}$.
\end{theorem}


\begin{thebibliography}{20}

\bibitem{BQ}
A. T. Benjamin and J. J. Quinn, \emph{Proofs that Really Count: The Art of Combinatorial Proof}, Mathematical Association of America, Washington, DC (2003).

\bibitem{Bl1}
A. Blecher, C. Brennan and A. Knopfmacher, Levels in bargraphs,
{\it Ars Math. Contemp.} 9 (2015), 287--300.

\bibitem{Bl2}
A. Blecher, C. Brennan and A. Knopfmacher, Peaks in bargraphs,
{\em Trans. Royal Soc. S. Afr.} 71 (2016),
97--103.

\bibitem{Bl3}
A. Blecher, C. Brennan and A. Knopfmacher, Combinatorial
parameters in bargraphs, {\em Quaest. Math.} 39 (2016),
619--635.

\bibitem{CMS}
S. Corteel, M. A. Martinez, C. D. Savage and M. Weselcouch, Patterns
in inversion sequences I, \emph{Discrete Math. Theor. Comput. Sci.} 18(2) (2016), Art. \#2.

\bibitem{GKP}
R. L. Graham, D. E. Knuth and O. Patashnik, \emph{Concrete Mathematics: A Foundation for Computer Science}, second edition, Addison-Wesley, Boston, MA (1994).

\bibitem{HT}
Q. Hou and T. Mansour, Kernel method and linear recurrence system, \emph{J. Comput. Appl. Math.} 261(1) (2008), 227--242.

\bibitem{KL}
D. Kim and Z. Lin, Refined restricted inversion sequences, \emph{S\'{e}m. Lothar. Combin.} 78B (2017), Art. \#52.

\bibitem{Lin}
Z. Lin, Restricted inversion sequences and enhanced 3-noncrossing partitions, \emph{European J. Combin.} 70 (2018), 202--211.

\bibitem{LY}
Z. Lin and and S. H. F. Yan, Vincular patterns in inversion sequences, \emph{Appl. Math. Comput.} 364 (2020), 124672.

\bibitem{MS}
T. Mansour and M. Shattuck, Pattern avoidance in inversion sequences, \emph{Pure
Math. Appl.} 25 (2015), 157--176.

\bibitem{MSa}
M. Martinez and C. Savage, Patterns in inversion sequences II: inversion sequences avoiding triples of relations, \emph{J. Integer Seq.} 21 (2018), Art. 18.2.2.

\bibitem{Par}
S. F. Parker, The combinatorics of functional composition and inversion, Ph.D. thesis, Brandeis U., 1993.

\bibitem{Sloane}
N. J. A. Sloane et al., The On-Line Encyclopedia of Integer Sequences, 2019. Available at https://oeis.org.

\bibitem{RS}
R. P. Stanley, \emph{Enumerative Combinatorics, Volume I}, Cambridge University Press, Cambridge, UK (1997).


\end{thebibliography}
\end{document}